                                                          \documentclass[9pt,twocolumn,technote,twoside]{IEEEtran}
\usepackage{subfigure}
\usepackage{graphicx}
\usepackage{graphics} 
\usepackage{epsfig} 
\usepackage{mathptmx} 
\usepackage{times} 
\usepackage{amsmath} 
\usepackage{amssymb}  
\usepackage{amsthm}
\usepackage{dsfont}
\usepackage{amssymb} 
\usepackage{color}

\everymath{\displaystyle}

\everymath{\displaystyle}
\newcommand{\beqnum}{\begin{equation}\begin{array}{lcl}}
\newcommand{\eeqnum}{\end{array}\end{equation}}
\newcommand{\beqnom}{\begin{eqnarray}}
\newcommand{\eeqnom}{\end{eqnarray}}
\newcommand{\beqnc}{\begin{center}\begin{eqnarray}}
\newcommand{\eeqnc}{\end{eqnarray}\end{center}}
\newcommand{\beqnlm}{\begin{equation}\vspace{-.5cm}\begin{array}{lll}}
\newcommand{\eeqnlm}{\end{array}\end{equation}}\vspace{-.5cm}
\newcommand{\beq}{\begin{eqnarray*}}
\newcommand{\eeq}{\end{eqnarray*}}
\newcommand{\bef}{\begin{figure}[tbh!]}
\newcommand{\enf}{\end{figure}}

\newcommand{\lva}{\left\lfloor}
\newcommand{\rva}{\right\rceil}
\newcommand{\vep}{\varepsilon}

\newcommand{\R}{\mathbb{R}}

\newtheorem{montheo}{\bf Theorem}
\newtheorem{rem}{\bf Remark}
\newtheorem{lemme}{\bf Lemma}
\newtheorem{proposition}{\bf Proposition}

\newtheorem{madef}{\bf Definition}
\newtheorem{notations}{\bf Notation}

\begin{document}
\title{Barrier Function-Based Adaptive Continuous Higher-Order Sliding Mode Controllers with Unbounded Perturbations}
\author{Yacine Chitour, Hussein Obeid, Salah Laghrouche,  and  Leonid Fridman
\thanks{Y. Chitour is with the Laboratoire des Signaux et Syst\`{e}mes, Universit\'{e} Paris-Saclay, Gif-sur-Yvette, Paris, France (e.mail:yacine.chitour@l2s.centralesupelec.fr).}
\thanks{H. OBEID is with LUSAC Laboratory, University of Caen Normandy, 14032 Caen, France (e.mail: hussein.obeid@unicaen.fr).}
\thanks{S. Laghrouche is with Femto-ST UMR CNRS, University of Bourgogne Franche-Comt\'{e}/UTBM, 90010, Belfort, France (e.mail: salah.laghrouche@utbm.fr).}
\thanks{L. Fridman is with the Departement of Robotics and Control, Engineering Faculty, Universidad Nacional Aut\'{o}noma de M\'{e}xico (UNAM), D.F 04510. M\'{e}xico (e.mail: lfridman@unam.mx).}
}

\date{}


\maketitle
\begin{abstract}
\noindent 
In this paper, two classes of continuous higher order adaptive sliding mode controllers based on barrier functions are developed for a perturbed chain of integrators with unbounded perturbations. Both classes provide finite-time convergence of system states to a predefined domain using a continuous control signal. The first class of adaptive controllers does not require any assumptions about the perturbation; however, it can provide unbounded control gains. To ensure bounded control gains, and in the case of a Lipschitz perturbation, a second class of adaptive controllers, called the adaptive higher order  Super-Twisting (HOST) algorithm, is developed. 
\end{abstract}






\section{Introduction}

Sliding Mode Control is an efficient tool for matched bonded uncertainties compensation. Higher Order Sliding Mode Controllers (HOSMC) (\cite{Levant2001,levant2003higher,levant2005homogeneity,CRUZZAVALA2017232,ding2016simple,zamora2013control,chalanga2015new,edwards2016adaptive,EDWARDS2016183,kamal2015new,CHL15,
moreno2018discontinuous,mercado2020discontinuous,cruz2020higher,seeber2021integral,mercado2021multiple}) allow collapsing in a finite time for the dynamics of the flat single input single output of order $r.$  The homogeneity properties of discontinuous HOSMCs ensure the $r$th order of asymptotic precision of the output with respect to the sampling step and parasitic dynamics (\cite{levant2010chattering,10.1002/rnc.4347,boiko2008discontinuous}).

Two types of HOSMCs have been are designed to compensate the perturbation in the chain of integrators of $r$th order. Discontinuous HOSMCs (DHOSMCs) have been introduced in the early 2000's (\cite{Levant2001,levant2003higher,levant2005homogeneity}) and are under active investigation (\cite{CRUZZAVALA2017232,ding2016simple}). The features of DHOSMCs are: (i) ability to compensate bounded matched perturbations theoretically exactly in finite-time;
(ii) ensuring a finite-time convergence to zero not only for the input but also for its derivatives till the order $(r -1)$ ($r$th order sliding set);
(iii) discontinuity of the control signal. For $r$th order perturbed chain of integrators governed by DHOSMCs, the homogeneity weight of output is $r$ and  $r-k, k=1,2,...,r-1$ for its $k$th derivatives. To implement DHOSMCs, the knowledge of perturbations upper bounds is needed.

Continuous HOSMCs (CHOSMCs) were proposed ten years later (\cite{zamora2013control,chalanga2015new,kamal2015new,edwards2016adaptive,CHL15,CRUZZAVALA2017232,
moreno2018discontinuous,mercado2020discontinuous,cruz2020higher,seeber2021integral,mercado2021multiple}). The main characteristics of CHOSMCs are: ability to compensate {\it Lipschitz} matched perturbations providing {\it {continuous}} control signal and theoretically exact finite time convergence to zero not only for the input but also for its derivatives till the order $r$ (($r+1$)th order sliding set). For $r$th order perturbed chain of integrators governed by CHOSM the homogeneity weight of output is $r+1$ and  $r-k+1, k=1,2,...,r$ for its $k$th derivatives.  To use CHOSMCs, the knowledge of the upper bound of
perturbations derivatives is also needed. A main problem in the use of CHOSMCs is the difficulty to estimate that upper bound. For example, to design a CHOSMC, the upper bound of the jerk is needed to be estimated, and, consequently, adaptation of CHOSMC gains is required.  

The homogeneity properties of HOSMCs ensure the asymptotic precision of the output and its derivatives with respect to the sampling step and parasitic dynamics (\cite{levant2010chattering,10.1002/rnc.4347,boiko2008discontinuous}). On the other hand, the amplitude of oscillations caused by non-idealities and energy consumption needed to maintain the system in real sliding mode (\cite{levant2010chattering,10.1002/rnc.4347}) depends on the upper bounds of the perturbations and their derivatives. Usually, such upper bounds are unknown or overestimated. That is why the papers devoted to HOSMCs adaptation (\cite{Bartolini_IMA2013,edwards2016adaptive,EDWARDS2016183,doi:10.1002/rnc.4105,doi:10.1002/rnc.4253,Mor2016ACTA,
OBEID2018540,OBEID2019STC,obeid2021dual,
rodrigues2021adaptive,Plestan2010,Plestan2012,yan2016adaptive,Shtessel}) should solve two contradictory problems simultaneously:

\begin{itemize}
  \item ensure finite-time convergence to the origin,
  \item Guarantee that the control signal follows the value of perturbation with opposite sign without overestimation. 
\end{itemize}

The main approaches to adaptation for sliding mode controllers can be split into three directions:

(a) the use of a filtered value of equivalent control
(\cite{edwards2016adaptive,EDWARDS2016183,Oliveira2018,utkin2013adaptive});

(b) increasing the gains (see, for example, \cite{ Mor2016ACTA});

(c) increasing and decreasing the gains (\cite{Bartolini_IMA2013,Plestan2010,Plestan2012}).\\

The approach (a) proposes to use the filtered value of the equivalent control as an estimation of the disturbance. The latter consists of increasing the gain to force the sliding motions. Once the sliding mode is achieved, the high frequency control signal is low pass filtered and used as information about the disturbance in the controller gain. The resulting sliding mode controllers gain is the sum of filtered signals and some constant to compensate possible errors between real disturbances and its value estimated by the low-pass filter. However, the algorithm (\cite{utkin2013adaptive}) requires the knowledge of the minimum and maximum allowed values of the adaptive gain; hence, it requires the information of the upper bound of disturbance derivatives. On the other hand, even if the other algorithms (\cite{edwards2016adaptive,EDWARDS2016183,rodrigues2021adaptive}) theoretically do not require the information of the disturbance's derivatives, however, in practice, the use of a low-pass filter requires implicitly information about this upper bound in order to choose the filter time constant. With this information CHOSMCs (\cite{CHL15,CRUZZAVALA2017232,moreno2018discontinuous,mercado2020discontinuous,cruz2020higher,seeber2021integral,mercado2021multiple}) can be applied and the adaptation will not be needed. 

The approach (b) consists of increasing the gain till the value ensuring that sliding motions are reached, then the gain is fixed at this value, ensuring an ideal sliding mode for some interval. This second strategy has two main disadvantages: (1) the gain does not decrease, i.e. it will not follow the disturbance when it is possibly decreasing and, correspondingly, it will cause big energy losses; (2) one cannot be sure that the sliding mode will never be lost because it cannot be assured that the disturbance will not grow anymore. 

To overcome the first of these disadvantages, the approach (c) was developed:  the sliding mode controller gain increases till the value ensuring that sliding motions are achieved, and then they decrease till the moment, when the sliding motion lost. Then, the procedure is repeated. So, the approach (c) ensures the solution of the perturbed system will be uniformly bounded but the size of the upper-bound cannot be predefined a priori because it depends on unknown upper bound of perturbations.

To reach the finite-time convergence to a prescribed vicinity of the sliding set, a barrier function (BF) based-adaptive strategy is proposed in \cite{OBEID2018540,OBEID2019STC}. Recently, this strategy has been applied to design an adaptive DHOSMC for a perturbed chain of integrators of order $r$ \cite{laghrouche2021barrier}. This algorithm ensures the finite-time convergence of the states to a predefined decreasing domain which tends to zero as $t$ tends to infinity. However, this algorithm gives rise to chattering when the sliding modes appear, moreover, it requires the assumption that the perturbation is bounded.  




In this paper, BF adaptation methodologies are extended to the classes of perturbed chain of integrators with unknown control gains and unbounded perturbations. Two different classes of continuous BF adaptation-based control algorithms are considered for two different classes of unbounded perturbations.

 The first class can be viewed as a generalization of the adaptive algorithm presented in \cite{laghrouche2021barrier}. Indeed, it ensures the same objective as before with the advantage that it can be applied for a more general case (unbounded perturbation), i.e. fewer restrictions on the bound of perturbation. Furthermore, it provides a continuous control signal (which is not the case in \cite{laghrouche2021barrier}). However, to maintain the above convergence with an unbounded disturbance, the control gain can tend to infinity.

To overcome this limitation, a second class of adaptive algorithms is developed for a perturbed chain of integrators of length $r$ with Lipschitz perturbation. For such a case, we propose the BF-based adaptation of the Higher-Order Super-Twisting (HOST) algorithm presented in \cite{HLC2017}. The main feature of 
this algorithm is that it ensures the finite-time convergence of the states to a predefined vicinity of the origin with a bounded control gain rather than an unbounded one. It is important to stress that the generalization of adaptive Super-Twisting for an arbitrary order is still an open problem. Indeed, the exciting adaptive ST algorithms are limited to order 1 and 2.  \cite{OBEID2019STC,Shtessel,edwards2016adaptive}. 



The paper is organized as follows. In Section~\ref{sec:sec2}, the problem statement is presented. In Section~\ref{sec:sec3}, some preliminary results are given. The first class of adaptive continuous HOSM controllers is designed in Section~\ref{sec:case1}. Then, the adaptation of the HOST algorithm handling the second class, is presented in Section~\ref{sec:case2}. Finally, some concluding remarks are given in Section~\ref{sec:conclusion}.

\section{Problem statement} \label{sec:sec2}
Consider the perturbed chain of integrators  given by
\begin{equation}\begin{array}{lcl}\label{nonlinear1}
\dot z_i=z_{i+1},~ i=1,...,{r-1},\quad \dot z_r =\gamma(t)u(t)+\phi(t),
\end{array}\end{equation}
where $r$ is a positive integer, $z=[z_1\  z_2\ ... z_r ]^T\in\mathbb{R}^r$ is the state vector, $u\in\mathbb{R}$ represents the scalar control variable and the functions $\gamma$ and $\phi$ are unknown measurable functions on $\mathbb{R}_+$.

The control objective is to design a continuous higher order sliding mode (HOSM) controllers to force the trajectory $z(\cdot)$ of System (\ref{nonlinear1}) to stay in an arbitrarily predefined neighborhood of zero once it reaches it, for the following two sets of conditions on the functions $\gamma$ and $\phi$:

 \begin{description}
\item[{\bf Case $1$:}]
\begin{equation}\label{bound-G}
0 < \gamma_m\le \gamma(t), \quad
\left| \phi(t) \right| \le \phi_M\tilde{\phi}(t),\hbox{ for a.e. }t\geq 0,
\end{equation}
\end{description}
where $\tilde{\phi}$  is a {\bf known} continuous function defined on $\R_+$, 
which is non-decreasing and lower bounded by $1$, and $\gamma_m$ and $
\phi_M$ are {\bf unknown} positive constants;
\begin{description}
\item[{\bf Case $2$:}]
\begin{equation}
\begin{array}{lcl}\label{bound-dot-psi}
\hspace{-1cm}\gamma\equiv \gamma_m, \quad \left| \dot \psi(t) \right| \le \psi_M,\hbox{ for a.e. }t\geq 0,
\hbox{ where }\psi:=\frac{\phi}{\gamma},
\end{array}
\end{equation}
\end{description}
with $\gamma_m$ and $\psi_M$ are {\bf unknown} positive constants. In both cases, the feedback controller we propose is of the form $L(t,z)u_r(z)$, where the gain function 
$L(t,z)$ is continuous and trajectory-dependent and the feedback function $u_r(\cdot)$
is a feedback arising from the unperturbed case, i.e. $\gamma\equiv \gamma_m$ and $\psi\equiv 0$. Even though Case $2$ is a subclass of Case $1$, we will design for it a 
feedback controller with bounded gain $L$, different from the feedback controller 
provided on Case $1$, which, in general, exhibits an unbounded gain $L$.






 To solve these two problems, the concept of barrier function will be used to design adaptive 
 continuous HOSM controllers ensuring the finite-time convergence of solutions to a desired 
 vicinity of origin for each case.

\begin{rem}
In the literature of discontinuous HOSM algorithms, the functions $\gamma$ and $\phi$ satisfy the following standard assumption \cite{levant2003higher,ding2016simple,CRUZZAVALA2017232,laghrouche2021barrier}
\beqnum\label{H1}
\quad 0<\gamma_m\leq \gamma(t)\leq\gamma_M,\quad  \left| \phi(t) \right| \le \phi_M,\hbox{ for a.e. }t\geq 0,
\eeqnum
with $\gamma_m$, $\gamma_m$ and $\phi_M$ are positive constants. However, in Case 1 the upper bound of $\phi(t)$ is a non-decreasing function rather than a positive constant and no upper bound of $\gamma$ is needed. Consequently, Case 1 can be viewed as a generalization of the standard assumption \eqref{H1}. 
\end{rem}

\begin{rem}
Note that Case 2 is similar to the conventional assumption given in the literature of continuous HOSM algorithms \cite{CHL15}. Indeed, the only difference is that the function $\gamma$ is supposed to be constant rather than a positive function. 
\end{rem}


\section{Preliminaries} \label{sec:sec3}


\begin{madef}\label{homogeneity} ({\bf Homogeneity})
If $r$ is a positive integer, a function $f:\mathbb{R}^r\rightarrow \mathbb{R}$ (respectively a vector field $F$ on $\mathbb{R}^r$) is said to be homogeneous of degree $q\in\mathbb{R}$ with respect to the family of dilations $(\delta_\varepsilon(z))_{\varepsilon>0}$, defined by
$$
\delta_\varepsilon(z)=(z_1,\cdots,z_r)\mapsto (\varepsilon^{p_1}z_1,\cdots,\varepsilon^{p_r}z_r),
$$
where $p_1,\cdots,p_r$ are positive real numbers (the weights), if for every $\varepsilon>0$ and $z\in\mathbb{R}^r$, one has $f(\delta_\varepsilon(z))=\varepsilon^qf(z)$  
(respectively  $\big(F(\delta_\varepsilon(z))=\varepsilon^q\delta_\varepsilon(F(z))\big)$).
\end{madef}

\noindent{\bf Notations:} We define the function $ \hbox{sgn}$ as the multivalued function defined on $\mathbb{R}$ by $ \hbox{sgn}(x)=\frac x{\vert x\vert}$ for $x\neq 0$ and $ \hbox{sgn}(0)=[-1,1]$. Similarly, for every $a\geq 0$ and $x\in \mathbb{R}$, we use $\left\lfloor x\right\rceil^a$ to denote $\left| x \right|^a  \hbox{sgn}(x)$. Note that $\lfloor \cdot\rceil^a$ is a  continuous function for $a>0$ and of class $C^1$ with derivative equal to $a\left| \cdot \right|^{a-1}$ for $a\geq 1$.  If $V:\R^r\rightarrow \R$ is a differentiable mapping, we use $\partial_jV$ and $\nabla V$ to denote the partial derivative of $V$ with respect to the $j$-th coordinate $z_j$ and the gradient vector of $V$, i.e. $\nabla V=(\partial_jV)_{1\leq j\leq r}$.
In the paper $\langle\cdot,\cdot\rangle$ denotes the standard inner product in $\mathbb{R}^r$.

We use $(e_i)_{1\leq i\leq r}$ to denote the canonical basis of $\R^r$ and $J_r$ the $r$-th Jordan block, i.e. the $r\times r$ matrix verifying $J_re_i=e_{i-1}$ for $1\leq i\leq r$ with the convention that $e_0=0$.
\\

We next provide the definition of a {\it pure chain of integrators of order $r$}.\begin{madef}\label{def0} {\it Let $r\geq 2$ be a positive integer. The $r$-th order pure chain of integrator
$(CI)_r$ is the single-input control system given by 
\beqnum \label{pure_int}
(CI)_r\ \ \ \dot z(t)=J_rz(t)+u(t) e_r,
\eeqnum
with $z=(z_1,\cdots,z_r)^T\in\mathbb{R}^r$ and $u\in \mathbb{R}$.
}
\end{madef}
We also give below the definition of an {\it adapted homogeneous feedback pair} for a pure chain of integrators of order $r$. 
\begin{madef}\label{def1}
{\it We say that $(V,u_r)$ with $V:\R\to\R_+$ and $u_r:\R^r\to\R$ is an adapted homogeneous feedback pair (AHFP) for a pure chain of integrators of order $r$ if there exist $\kappa\in (-1,0)$ and $p\in (0,2)$ such that the following items hold true.
\begin{description}
\item[$(i)$] if $p_i=p+(i-1)\kappa$ for $1\leq i\leq r+1$, then $p_{r+1}>0$;
\item[$(ii)$] $u_r$ is continuous with $u_r(0)=0$ and homogeneous of degree $p_{r+1}$ with respect to $(\delta_\vep)_{\vep>0}$;
\item[$(iii)$] $V$ is a $C^1$ strict Lyapunov function along the 
pure chain of integrators of order $r$ and homogeneous of degree $2$ with 
respect to $(\delta_\vep)_{\vep>0}$ such that there exists a continuous 
function $\rho:\R^r\to\R$ bounded from below and from above by two 
positive constants $c_r$ and $d_r$ for which the time derivative of 
$V$ along non-trivial trajectories of $\dot z=J_rz+u_r(z)e_r$ verifies 
\beqnum \label{est-V1}
\dot V(z) =-\rho(z)V(z)^{1+\frac{\kappa}2};
\eeqnum
\item[$(iv)$] the function $z\mapsto u_r(z)\partial_r V$ is continuous,
non-positive over $\mathbb{R}^i$.
\end{description}
If Item $(iv)$ is replaced by the stronger requirement that 
$u_r=-l_r\lva \partial_r V\rva^{\gamma_r}$, 
for some $\gamma_r>0$,
the homogeneous feedback pair $(V,u_r)$ is said to be strongly adapted (sAHFP).
}
\end{madef}

\begin{rem}\label{rem:imme0}
It has first to be noted that (sAHFP) for a pure chain of integrators of order 
$r$ exist, according to the fundamental work \cite{Hong}, see also \cite{HLC2017} for 
other examples of (AHFP). In particular, $\gamma_r=\frac{p_{r+1}}{2-p_r}$. 
Moreover, since the exponent of $V$ in the right-hand side of \eqref{est-V1} 
is less than one, it follows the stabilization of a pure chain of integrators with 
an (AHFP) occurs in finite time. Finally, the integer $2$ in the definition can 
be replaced by any positive real number $q$ with appropriate modifications 
in subsequent inequalities. 
\end{rem}

We collect in the following lemma basic and useful facts about a (sAHFP).
\begin{lemme}\label{lem:basic1}
With the notations of Definition~\ref{def1}, it is immediate to get that 
\begin{description}
\item[$(a)$] the vector fields $z\mapsto J_rz$ and $z\mapsto u_r(z)e_r$ are both homogeneous of degree $\kappa$ with respect to $(\delta_\vep)_{\vep>0}$;
\item[$(b)$] the function $\langle \nabla V(z),J_rz+u_r(z)e_r\rangle$ is homogeneous  with respect to $(\delta_\vep)_{\vep>0}$ of degree $2+\kappa$. As a consequence, there exists $c_u>0$ such that 
\begin{equation}\label{eq:ur-pVr-Z}
\vert u_r(z)\partial_rV(z)\vert\leq c_u V^{1+\frac{\kappa}2},\quad \forall z\in\R^r.
\end{equation}
\item[$(c)$] for every $\varepsilon>0$ and $z\in\R^r$, it holds
\begin{equation}\label{eq:grad}
\delta_\varepsilon\big(\nabla V(\delta_\varepsilon(z))\big)=\varepsilon^2\nabla V(z);
\end{equation}
\item[$(d)$] Set $D_p=\mathrm{diag}(p_1,\cdots,p_r)$. Then it holds, for every 
$z\in\R^r$,
\begin{equation}\label{eq:grad}
\langle \nabla V(z),D_pz\rangle=2V(z).
\end{equation}
\end{description}
\end{lemme}
Item $(d)$ is the classical Euler relation verified by the $C^1$ homogeneous real valued function $V$.
\begin{notations} In the sequel, several constants will appear, which are 
independent of a particular trajectory of \eqref{nonlinear1} but which may 
depend on some of the bounds appearing in \eqref{H1} and also possibly on 
the parameter $\varepsilon$ which will be used to define the target to reach 
in finite time. To track this dependence, we will repeatedly use the symbol 
$C_j(E)$, where $j$ is an integer and $E$ is any subset of $\{\gamma,\psi,\phi\}$. 
If $E=\emptyset$ (i.e. the constant does not depend on the unknown 
constants in $\{\gamma,\psi\}$), we simply use $C_j$.
\end{notations}

\section{Case 1}\label{sec:case1}
\subsection{ Adaptive controller design}

\noindent Consider System \eqref{nonlinear1} where the uncertainty functions 
$\gamma(\cdot)$ and $\phi(\cdot)$ are measurable and locally bounded, and satisfy \eqref{bound-G} with {\bf unknown} constants $\gamma_m$ and $\phi_M$.
Moreover, several constants appearing later on in the form $C_j(E)$
with $E$ any subset of $\{\gamma,\phi\}$ only depend on $\{\gamma_m,\phi_M\}$.


To solve Problem 1, we need to define appropriate neighborhoods of origin where the state $z$ is supposed to enter and then stay there for all subsequent times. Such neighborhoods are defined, for $\mu>0$ by 
$
{\cal{V}}_\mu:=\{z\in\R^r\mid V(z)<\mu\},
$
where $V$ is part of some given (AHFP) for the pure chain of integrators of order $r$ as given in Definition~\ref{def1}.
We then propose the following adaptive controller
\beqnum\label{ST-feedback}
 u(t) =   L(t,z)u_r(z(t))\\
\eeqnum
where the adaptive gain $L(t,z)$ will be chosen later on.

Plugging \eqref{ST-feedback} into \eqref{nonlinear1} yields,
\beqnum\label{eq:ST-feedback}
\dot z(t) =  J_rz(t)+\gamma\ e_rL(t,z)u_r(z(t))+\phi(t),
\eeqnum
as long as such trajectories are defined.

From now on, we fix an arbitrary function $\mu:[0,\infty)\to (0,1)$ of class $C^1$ which is decreasing to zero, verifying 
\begin{equation}\label{eq:mu1}
\dot \mu(t)>-\frac{c_r}2\mu^{1+\frac{\kappa}2}(t),\quad \forall t\geq 0.
\end{equation}
Note that the $C^1$ function $\tilde\alpha$ defined by 
\begin{equation}\label{eq:alphatilde}
\tilde\alpha(t)=\frac{\mu(t)}{\tilde{\phi}(t)^{1+\frac1{\gamma_r(1+\frac{\kappa}2)}}},\ \forall t\geq 0
\end{equation}
is non-increasing.

In order to define the adaptive gain $L(t,z)$ inside ${\cal{V}}_{\mu(t)}$, $t\geq 0$, we need to consider, for any $\mu>0$, the function 
 $L_\mu$ is defined as 
\begin{equation}\label{eq:L}
L_\mu [0,\mu)\to (0,\infty),\quad
V\mapsto \frac{\mu}{\mu-V}.
\end{equation}
On the other hand, for the part of the trajectory outside ${\cal{V}}_{\mu(t)}$, 
we must consider a time-varying non-decreasing continuous gain function 
$l:\R_+\to [1,\infty)$ only subject to 
\begin{equation} \label{eq:lt}
\lim_{t\to \infty}\frac{l(t)\mu(t)^{\gamma_r(1+\frac{\kappa}2)}}{\tilde{\phi}(t)^{1+\gamma_r(1+\frac{\kappa}2)}} = \infty.
\end{equation}

The adaptive gain $L(t,z)$ is defined as follows: it is equal to $l(t)$ as long as $V (z(t))> \mu(t)/2$ 
and then, if there exists a first time $\bar t$ so that $V (z(t))\le \mu(t)/2$, the adaptive gain 
$L(t,z)$ is taken equal to $\bar{c}L(z)$, for $t>\bar t$, where 
$\bar{c}=\frac{l(\bar t)}{2^{\gamma_r(1+\frac{\kappa}2)}}\geq 1$.

Hence, the variable gain $L(t,z)$ is defined, as long as the trajectory $z(\cdot)$ of \eqref{eq:ST-feedback} exists, by  
\begin{equation}
\label{eq:adaptgain3}
L(t,z)=\begin{cases}
l(t),  \quad \quad \textrm{if} \quad 0\leq t < \bar{t},  \\
 \bar{c}L_{\mu(t)}^{\gamma_r(1+\frac{\kappa}2)}(z(t)), \quad \textrm{if}  \quad t\geq\bar{t},
\end{cases}
\end{equation} %
with the convention that $\bar{t}=0$ if 
$V(z_0)\le \mu(\bar t)/2$ and $\bar{t}=\infty$ if $z(\cdot)$ is defined 
for all times and $V(z(t))>\mu(t)/2$. Note that $\bar{c}$ has been 
chosen so that the time-varying feedback $t\mapsto L(t,z(t))$ is continuous.

We next provide our main stabilization result, whose proof is deferred to the next subsection.

\begin{montheo}
\label{theorem1} 
Let $r$ be a positive integer and System \eqref{nonlinear1} be the perturbed 
$r$-chain of integrators with uncertainty functions $\gamma(\cdot)$ and 
$\phi(\cdot)$ subject to Eqs. \eqref{bound-G}, where 
the constants $\gamma_m$ and ${\phi}_M$ are {\bf unknown}. Let 
$(u_r,V)$ be a (sAHFP) for the pure chain of integrators of order $r$ as given 
in Definition~\ref{def1}. For every $z_0\in\mathbb{R}^r$, consider any 
trajectory $z(\cdot)$ starting at $z_0$ of System \eqref{nonlinear1} closed by 
the feedback control law \eqref{ST-feedback} where the adaptive gain 
$L(t,z)$ is given in \eqref{eq:adaptgain3}. Then,  
\begin{description}
\item[$(i)$] the trajectory $z(\cdot)$ is defined for all non-negative times; 
\item[$(ii)$] there exists a first time $\bar{t}$ for which $V(z(\bar{t}))
\le{ \mu(\bar t)/2}$ and for all $t\ge \bar{t}$, one has $V(z(t))<\mu(t)$;
\item[$(iii)$]  there exists a time $t_0$ for which 
\begin{equation}\label{eq:Vfinal}
V(z(t))\leq \big(1-C_1(\gamma,\phi)\tilde\alpha(t)\big)\mu(t),\ \forall t\geq t_0.
\end{equation}
\end{description}
\end{montheo}
\begin{rem}\label{rem:unique} 
Since there is no uniqueness for the trajectories of System \eqref{nonlinear1} 
closed by the feedback control law \eqref{ST-feedback}, where the adaptive 
gain $L(t,z)$ is given in \eqref{eq:adaptgain3}, and starting at $z_0$, the time 
$\bar{t}$ does depend on that trajectory and not just on $V(z_0)$.
\end{rem}
 \subsection{Proof of Theorem~\ref{theorem1}}
 Before proceeding with our argument, we need the following technical result, which will be used repeatedly in the 
sequel.
\begin{lemme}\label{lem:urVr} 
There exist positive constants $C_0,C_1$ 
such that, for every $\eta>0$, one has
\begin{equation}\label{eq:urVr}
|\partial_rV|\leq C_0\eta^{1+\frac{\kappa}2}+C_1\frac{|u_r\partial_rV|}{\eta^{\gamma_r(1+\frac{\kappa}2)}},
\end{equation}
with $\gamma_r>0$ given by the (sAHFP) condition. 
\end{lemme}
Such an inequality follows by applying Young's inequality to 
$\vert\partial_rV\vert$, which in turn is obtained from the definition of a (sAHFP).

For simplicity, we denote $V(z)$ by 
$Z$ and drop the (obvious) arguments in several quantities below. 

We have the following technical lemma.
\begin{lemme}\label{le:firstproof}
Let $z(\cdot)$ be a non-trivial trajectory of System \eqref{nonlinear1} closed 
by the feedback control law \eqref{ST-feedback} where the adaptive gain 
$L(t,z)$ is given in \eqref{eq:adaptgain3}. As long as $z(\cdot)$
is defined, one has that  
\begin{eqnarray}
\frac{dZ}{dt}&\leq&-c_r\Big(Z^{1+\frac{\kappa}2}-\Big(\frac{\mu}3\Big)^{1+
\frac{\kappa}2}\Big)\nonumber\\
& &-|u_r\partial_rZ|\gamma_m\Big(L(t,z)-\tfrac1{\gamma_m}-\frac{C_4(\gamma,
\phi)\tilde{\phi}^{1+\gamma_r(1+\frac\kappa2)}}{\mu^{(1+\frac{\kappa}2)\gamma_r}}\Big).\label{eq:lemma32}
\end{eqnarray}
\end{lemme}
\begin{proof} 
By taking the time derivative of $Z(t)$ along any non-trivial trajectory of \eqref{eq:ST-feedback11}, one gets
\begin{eqnarray}
\frac{dZ}{dt}&=&\langle\nabla Z,J_rz+e_ru_r\rangle
+\vert u_r\partial_rZ\vert-\gamma\vert u_r\partial_rZ\vert L(t,z)+
\phi\partial_rZ\nonumber\\
&\leq& -\rho Z^{1+\frac{\kappa}2}-|u_r\partial_rZ|\gamma_m\Big(L(t,z)-\tfrac1{\gamma_m}\Big)+\phi_M\tilde{\phi}\vert\partial_rZ\vert.
\label{eq:key}
\end{eqnarray}
We next use Lemma~\ref{lem:urVr} (to the term 
$\vert\partial_rZ\vert$ in \eqref{eq:key}) with the parameter $\eta>0$ now 
equal to $\eta\tilde{\mu}(t)$ with $\eta\in (0,1)$ and $\tilde{\mu}(t)>0$ chosen 
so that 
$$
\tilde{\phi}(t)\tilde{\mu}(t)^{1+\frac{\kappa}2}=\mu(t)^{1+\frac{\kappa}2},\ \forall t\geq 0.
$$
We obtain that
\begin{eqnarray}
\frac{dZ}{dt}&\leq&-c_r\Big(Z^{1+\frac{\kappa}2}-C_2(\phi)(\eta\mu)^{1+
\frac{\kappa}2}\Big)\nonumber\\
& &-|u_r\partial_rV|\gamma_m\Big(L(t,z)-\tfrac1{\gamma_m}-\frac{C_3(\gamma,
\phi)\tilde{\phi}^{1+\gamma_r(1+\frac\kappa2)}}{(\eta\mu)^{(1+\frac{\kappa}2)\gamma_r}}\Big),\label{eq:lemma31}
\end{eqnarray}
for every $\eta\in (0,1)$. Choosing the latter parameter \eqref{eq:lemma31} according to  
$$
C_2(\phi)(3\eta)^{1+\frac{\kappa}2}=1,
$$
one gets \eqref{eq:lemma32}. 
\end{proof}
We now prove the existence of a finite $\bar{t}$ for the trajectory $z(\cdot)$ starting at $z_0$ of System 
\eqref{nonlinear1} closed by the feedback control law \eqref{ST-feedback} where the adaptive gain $L(t,z)$ is 
given in \eqref{eq:adaptgain3}. From the definition of $L$, it remains to do so when $V(z_0) > \mu(0)/2$ and in this case, $L(t,z)=l(t)$ as long as the corresponding trajectory is defined so that 
$V(z(t)) > \mu(t)/2$. For these times, one can write \eqref{eq:ST-feedback} as 
\beqnum\label{eq:ST-feedback11}
\dot z(t)=  J_rz(t)+\gamma\ e_rl(t)u_r(z(t))+\phi(t),\\
\eeqnum
as long as such trajectories are defined. 

We actually start by showing Item $(ii)$ o the main theorem. 
\begin{lemme}\label{le:existence1}
With the notations above, the time $\bar t$ exists for every trajectory of 
\eqref{eq:ST-feedback11} is defined and finite.
\end{lemme}
\begin{proof}
By the choice of $l(t)$ given in \eqref{eq:lt},  the last term in parentheses in
\eqref{eq:lemma32} tends to infinity as $t$ tends to infinity and one finally gets from \eqref{eq:lemma32} that there exists $t_1>0$ such that
$$
\frac{dZ}{dt}\leq-{\bar c}_rZ^{1+\frac{\kappa}2},
$$
as long as $t\geq t_1$ and $Z>\frac{\mu(t)}2$. That equation shows that the 
trajectory of \eqref{eq:ST-feedback11} is well defined as long as these 
conditions hold and it must stop in finite time otherwise one would get, 
because $1+\frac{\kappa}2<1$, convergence to zero in finite time, which 
yields an obvious contradiction. That concludes the proof of
Lemma~\ref{le:existence1}.
\end{proof}
From now on, $t\geq \bar{t}$, the adaptive gain $L(t,z)$ is equal to 
$\bar{c}L_{\mu(t)}$ and $z(t)\in {\cal{V}}_{\mu(t)}$ as long as that trajectory is defined. 
We start by writing \eqref{eq:ST-feedback} as 
\beqnum\label{eq:ST-feedback1}
\dot z(t)&= & J_rz(t)+\gamma\ e_r\bar{c}L_{\mu(t)}(Z(t))u_r(z(t))+\phi(t),
\eeqnum
as long as such trajectories are defined for $t\geq \bar{t}$. 
We first prove that trajectories of \eqref{eq:ST-feedback1} are defined for all 
$t\geq \bar{t}$.
\begin{lemme}\label{le:existence}
With the notations above, trajectories of \eqref{eq:ST-feedback1} are 
defined for all $t\geq \bar{t}$.
\end{lemme}
\begin{proof}
We argue by contradiction and assume that there exists such a trajectory 
defined on $[\bar{t},t_*)$ with $t_*> \bar{t}$ finite so that $Z(t_*)=\mu(t_*)$ 
and $Z(t)<\mu(t)$ on $[\bar{t},t_*)$. In particular, this implies that
\begin{equation}\label{eq:contra}
\int_t^{t_*}\big(\dot Z(s)-\dot\mu(s)\big)ds> 0,
\end{equation}
for $t$ in a left neighborhood of $t_*$.

Since $Z(t)< \mu(t)$ on $[\bar{t},t_*)$, the continuous functions 
$u_r\partial_r V$ and $\partial_r V$ remain bounded on $[\bar{t},t_*)$ and 
admit finite limits at $t=t_*$. From \eqref{eq:key} with 
$L(t,z)=\bar cL_{\mu(t)}$, one deduces that $\partial_r Z$ must tend to zero as $t$ tends to $t_*$ 
for otherwise $\dot Z-\dot\mu$ would be negative in a left neighborhood of  
$t_*$, contradicting \eqref{eq:contra}. It follows that
$$
\dot Z(t_*)-\dot\mu(t_*)\leq -c_r\mu^{1+\frac{\kappa}2}(t_*)-\dot\mu(t_*)<0,
$$
(where we have used \eqref{eq:mu1}), which again contradicts 
\eqref{eq:contra}. This concludes the proof of the 
lemma.
\end{proof}
We now proceed in showing that every trajectory of 
\eqref{eq:ST-feedback1} ultimately verifies the bound given in 
Theorem~\ref{theorem1} and this is the content of the following proposition.
\begin{proposition}\label{prop:fdl1}
Let $\tilde{\alpha}$ be the continuous function defined in \eqref{eq:alphatilde}.
Then, there exists a time $t_0\geq 0$ for which 
\begin{equation}\label{eq:bddXxi}
Z(t)\leq \Big(1-C_1(\gamma,\phi)\tilde\alpha(t)\Big)\mu(t) \quad \hbox{ for }t\geq t_0.
\end{equation}
\end{proposition}
\begin{proof} From \eqref{eq:lemma32} with 
$L(t,z)=\bar{c}L_{\mu(t)}^{\gamma_r(1+\frac{\kappa}2)}(z(t))$, one deduces 
that there exists $C_1(\gamma,\phi)$ so that, if 
$Z(t)\geq \big(1-C_1(\gamma,\phi)\tilde\alpha(t)\big)\mu(t)$, then 
$$
\bar{c}L_{\mu(t)}^{\gamma_r(1+\frac{\kappa}2)}(z(t))-\tfrac1{\gamma_m}-\frac{C_4(\gamma,
\phi)\tilde{\phi}^{1+\gamma_r(1+\frac\kappa2)}}{\mu^{(1+\frac{\kappa}2)\gamma_r}}<0.
$$
Using that fact in \eqref{eq:lemma32} and by taking into account \eqref{eq:mu1} and \eqref{eq:alphatilde}, one deduces that, if $Z(t)\geq \big(1-C_1(\gamma,\phi)\tilde\alpha(t)\big)\mu(t)$ for $t$ large enough, then
\begin{equation}\label{eq:theo1}\frac{d}{dt}\Big(Z-\big(1-C_1(\gamma,\phi)\tilde\alpha\big)\mu\Big)\leq
-\tfrac{c_r}2\Big(Z-\big(1-C_1(\gamma,\phi)\tilde\alpha\big)\mu\Big).
\end{equation}
We can now repeat the reasoning given in Lemma~\ref{le:existence1}
and Lemma~\ref{le:existence} with $Z-\big(1-C_1(\gamma,\phi)\tilde\alpha\big)\mu\Big)$ instead of $Z$ and one concludes.
\end{proof}
\begin{rem}\label{rem:gain0} If $\mu$ is constant and $\tilde\phi\equiv 1$, 
one recovers the standard assumption \eqref{H1} with two additional 
features: $(a)$ one does not need any upper bound on the function 
$\gamma$; $(b)$ the function 
$\tilde\alpha$ defined in \eqref{eq:alphatilde} is now constant (hence less 
than one). As a consequence of $(b)$, one deduces from \eqref{eq:Vfinal} that 
the adaptive gain function $L(t,z)=\bar{c}L_{\mu(t)}^{\gamma_r(1+\frac{\kappa}2)}(z(t))$ remains {\bf uniformly} 
bounded for large times by a positive constant only depending on 
$\gamma_m$ and $\phi_M$. However, if $\tilde\alpha$ tends to zero, then the upper bound for the adaptive gain furnished by \eqref{eq:Vfinal} tends to infinity.
\end{rem}

\begin{rem}
It should be mentioned that the proposed controller has two main advantages compared to the existing adaptive discontinuous HOSM controller proposed in \cite{laghrouche2021barrier}. Indeed, it can be used to solve more complex control problems and it provides a continuous control signal instead of a discontinuous one.   
\end{rem}

\begin{rem}The choice of the decreasing function $\mu$ is essentially only 
conditioned by \eqref{eq:mu1}. In particular, the latter condition allows one to 
possibly choose $\mu$ decreasing to zero in finite time. However, this 
convergence may occur before $Z(t)-\mu(t)/2$ becomes negative since in 
\eqref{eq:lemma32} if $z(\cdot)$ leaves in the sliding surface then $\mu(t)$ 
may be equal to zero. This would give a convergence of the state in finite 
time but with an infinite gain $l(t)$, which is not admissible. 
\end{rem}

\subsection{An example}
Consider the following third order disturbed system  
\beqnum \label{eq:systemchain3}
	\dot z_1 = z_2,\,
	\dot z_2 =  z_3,\,
	\dot z_3 = \phi + \gamma u.
\eeqnum
where the uncertainties $\gamma$ and $\phi$ are chosen as follows
\beqnum \label{eq:boundsuncer*}
	\phi = 3(1+4t), \, \phi_M=3 \,
	\gamma = 3 + 0.5 \,  (\sin(5t)), \, \gamma_m= 2.5.
\eeqnum

The objective is to ensure the convergence of the states $z_1$, $z_2$ and $z_3$ and maintain them in the predefined neighborhood of zero $V(z) < \mu=5 e^{-0.2t}$ using the adaptive continuous controller \eqref{ST-feedback} in the presence of the uncertainties \eqref{eq:boundsuncer*} whose bounds $\phi_M$ and $\gamma_m$ are unknown.

In the simulations, the initial values of the states are chosen as $z_1(0)=1$, $z_2(0)=1$ and $z_3(0)=-1$. The AHFP is designed according to Hong controller \cite{Hong} where the parameter values are set as $\kappa = -1/6$ and $p=1$. For a detailed description on the design of controller $u(t)$ see the Appendix  in \cite{laghrouche2021barrier}.  

\begin{figure}[h!]
\includegraphics[trim= 0.5cm 2.3cm 0.1cm 0.0cm, clip, width=9cm]{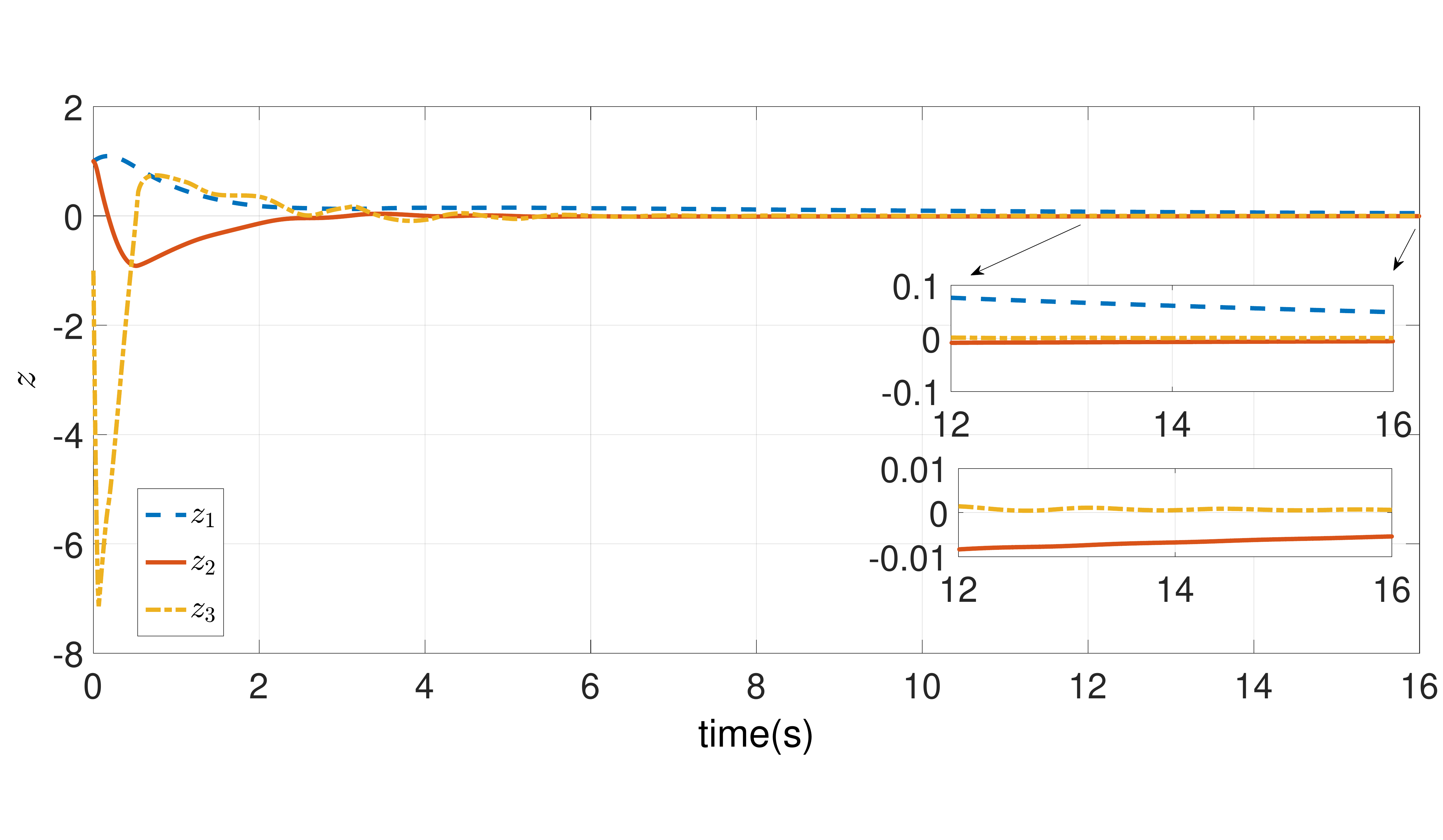}%
\caption{ The states $z_1$, $z_2$ and $z_3$ with the adaptive controller \eqref{ST-feedback}.}
\label{fig:statescase1}
\end{figure}

\begin{figure}[h!]
\includegraphics[trim= 2.5cm 2.3cm 4.1cm 3.0cm, clip, width=9cm]{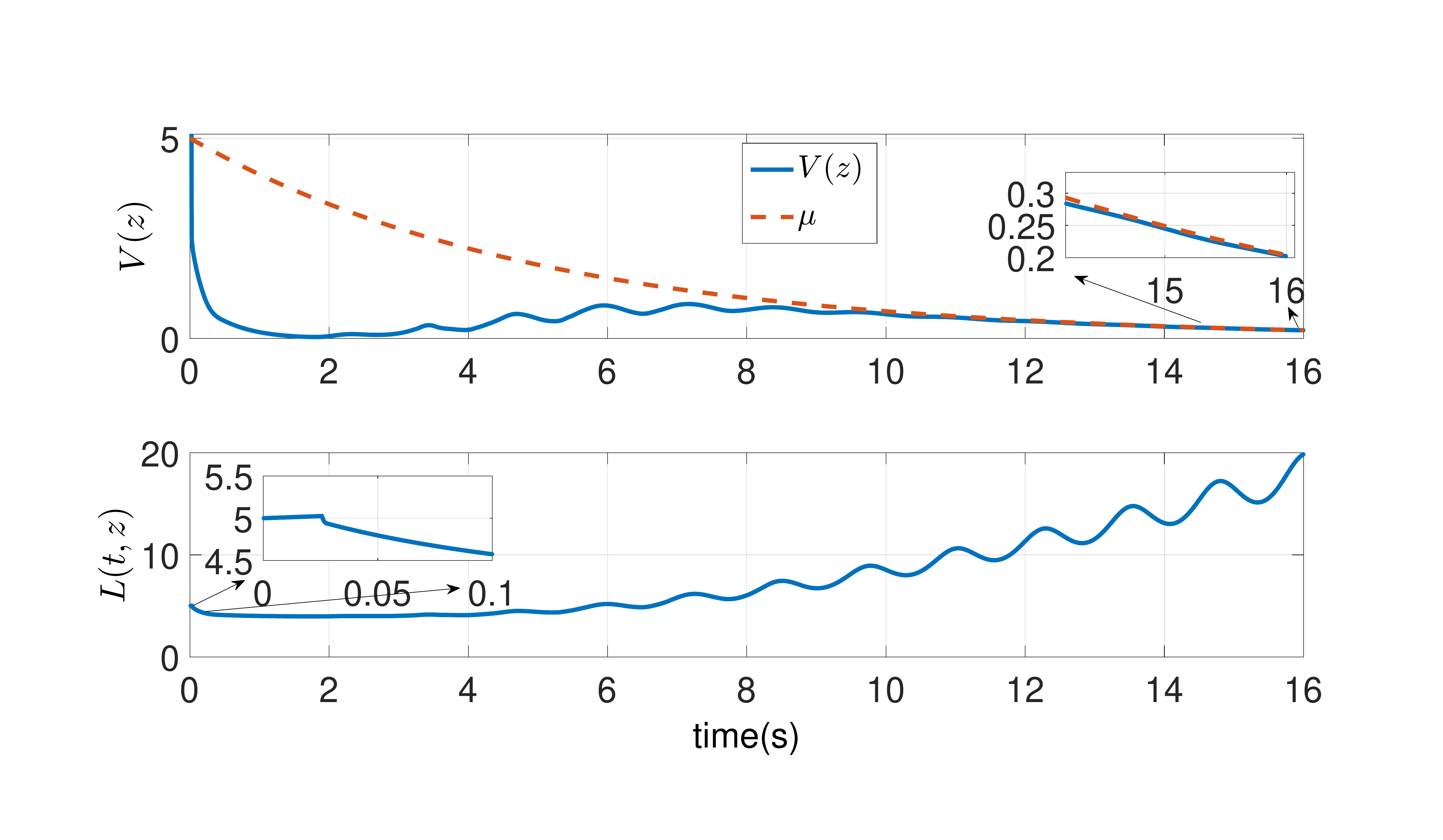}%
\caption{The Lyapunov function $V(z)$ and the adaptive gain $L(t,z)$ \eqref{eq:adaptgain3}.}
\label{fig:lyapunovcase1}
\end{figure}

\begin{figure}[h!]
\includegraphics[trim= 0.5cm 2.3cm 0.1cm 3.0cm, clip, width=9cm]{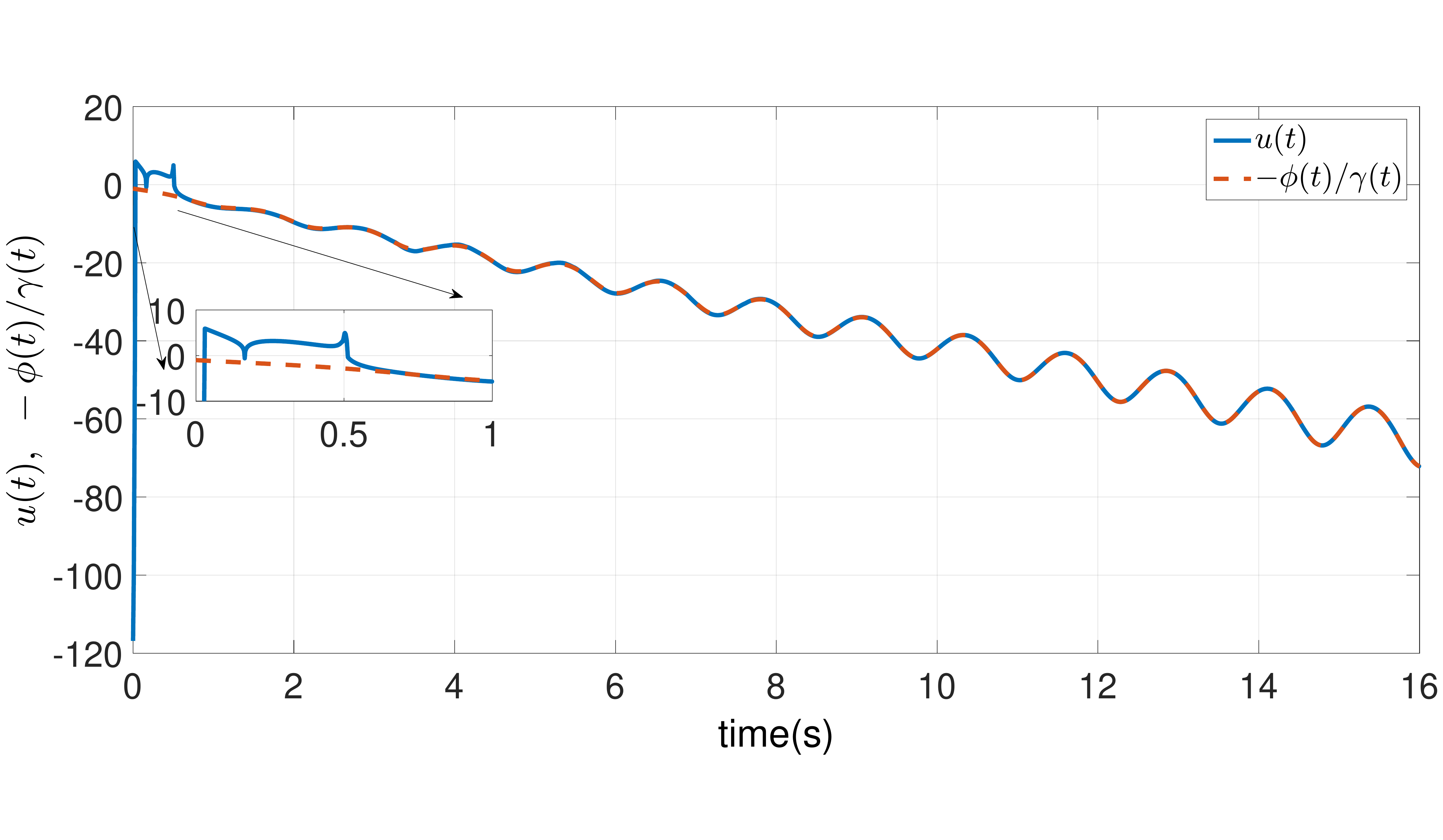}%
\caption{The control signal \eqref{ST-feedback} and the uncertainties $-\phi(t)/\gamma(t)$.}
\label{fig:estimcase1}
\end{figure}

The simulation results for system \eqref{eq:systemchain3} with the proposed adaptive controller are depicted in Figs.~\ref{fig:statescase1}-\ref{fig:lyapunovcase1}-\ref{fig:estimcase1}. Fig.~\ref{fig:statescase1} shows the convergence of the states $z_1$, $z_2$ and $z_3$ to some neighborhoods of zero. Fig.~\ref{fig:lyapunovcase1} confirms the achievement of the control objective, i.e. it shows that starting from the time instant $\bar{t}$, the states $z_1$, $z_2$ and $z_3$ stay in the predefined neighborhood of zero $V(z) < \mu$ despite the uncertainties. Moreover, it displays that the adaptive gain increases linearly until the time instant $\bar{t}$, then it follows the barrier function for all future time. It can be also noticed that the barrier function will tend to infinity due to the reason that $\mu(t)$ tends to zero and the function $\phi(t)$ tends to infinity (see Remark~\ref{rem:gain0}). The behaviors of the control signal and the uncertainties are depicted in Fig.~\ref{fig:estimcase1}. It can be clearly seen that the control signal is continuous and it follows the uncertainties starting from $\bar{t}$.

\section{Case 2 }\label{sec:case2}
\subsection{Adaptive controller design }

\noindent Consider the system described by \eqref{nonlinear1} where the 
uncertainty functions $\gamma(\cdot)$ and $\phi(\cdot)$ satisfy \eqref{bound-dot-psi}. In Section~\ref{sec:case1}, we have provided a continuous (HOSM) controller 
to force trajectories of System (\ref{nonlinear1}) to stay in an arbitrarily 
predefined neighborhood of zero once they reach it. However, as explained 
in Remark~\ref{rem:gain0}, trajectories may approach the boundary of that 
neighborhood arbitrarily close, which results in a possibly unbounded gain function $L$. 
In this section we bring a similar solution with a bounded gain function $L$ at the price of replacing Assumption \eqref{bound-G} by Assumption \eqref{bound-dot-psi}.  

As devised first in \cite{HLC2017}, the following Higher Order Super-Twisting (HOST)

\beqnum\label{ST-feedback*}
\begin{array}{ccl}
u_{ST}(t) &= &k_Pu_r(z(t)) + \bar{\xi}(t),\\
\dot{\bar{\xi}}(t)& =& -k_I \partial_r V(z(t)),\ \bar{\xi}(0)=0,\\
\end{array}
\eeqnum
where $(u_r,V)$ is a (AHFP) for the pure chain of integrators of order $r$ as given in Definition~\ref{def1}
and $k_P\geq \gamma_m$ and $k_I>0$ are arbitrary. However, the implementation of this HOST controller requires the information 
of the bounds $\gamma_m$ and ${\psi}_M$, 
which are not available anymore. 

Therefore, to overcome this problem,  we need to come up with an 
adaptative version of a HOST to solve this problem. 
In the sequel, we will assume for simplicity that both $k_P$ and $k_I$ are equal to one.

We propose the following adaptive controller
\beqnum\label{ST-feedback*}
\begin{array}{ccl}
u_{BST}(t) &= &  L_1(t,z)u_r(z(t))+ \xi(t),\\
\dot \xi(t)& =&- L_2(t,z)\partial_r V(z(t)),\ \xi(0)=0,\\
\end{array}
\eeqnum
where the adaptive gains $L_1(t,z)$ and $L_1(t,z)$ will be chosen later on.

Setting 
$
\phi=\xi+\psi,
$
and plugging \eqref{ST-feedback*} into \eqref{nonlinear1} yields,
\beqnum\label{eq:ST-feedback*}
\begin{array}{ccl}
\dot z(t)&= & J_rz(t)+\ e_r\Big(L_1(t,z)u_r(z(t))  + \phi(t)\Big),\\
\dot \phi(t)& =& -L_2(t,z) \partial_r V(z(t))+\dot\psi(t),\\
\end{array}
\eeqnum
as long as such trajectories are defined. Note that now $\phi(0)\in \R$ is 
arbitrary.

The beginning of the stabilization procedure is now similar to that of 
Case 1: after fixing an arbitrary $\varepsilon>0$, the variable gain $L_1(t,z)$ 
is defined, as long as the trajectory $z(\cdot)$ of \eqref{eq:ST-feedback*} 
exists, again by \eqref{eq:adaptgain3}, with the same convention on 
$\bar{t}(V (z_0))=0$. As for the variable gains $L_1(t,z)$ and $L_2(t,z)$, they 
are defined as 
\begin{equation}
\label{eq:adaptgain33}
L_1(t,z)=\begin{cases}
l(t),  \quad \quad \textrm{if} \quad 0\leq t < \bar{t},  \\
 \bar{c}L_{\varepsilon}^{-\kappa/2}
 (z(t)), \quad \textrm{if}  \quad t\geq\bar{t},
\end{cases}
\end{equation} 
and
\begin{equation}
\label{eq:adaptgain4}
L_2(t,z)=\begin{cases}
0,  \quad \quad \textrm{if} \quad 0\leq t < \bar{t},  \\
 L_{\varepsilon}(z(t)), \quad \textrm{if}  \quad t\geq\bar{t},
\end{cases}
\end{equation}
where $\bar{c}={l(\bar t)}/2^{-\frac{\kappa}2}\geq 1$ and the time $\bar t$ is defined as the first 
time $t$ so that $V (z(t))\le \varepsilon/2$ (with the same conventions 
as in the previous section).

\noindent It is immediate to see that $\xi\equiv 0$ on $[0,\bar{t}]$.

We next provide our main stabilization result, whose proof is deferred to the next subsection.

\begin{montheo}
\label{theorem1*} 
Let $r\geq 2$ be a positive integer, $\varepsilon>0$ and System 
\eqref{nonlinear1} be the perturbed $r$-chain of integrators with 
$\gamma(\cdot)\equiv \gamma_m$ and ${\psi}_M$ be the (unknown) bounds 
on $\gamma(\cdot)$ and $\dot \psi(\cdot)$. Let $(u_r,V)$ be an (AHFP) for 
the pure chain of integrators of order $r$ as given in Definition~\ref{def1}.
For every $z_0\in\mathbb{R}^r$, consider any trajectory $z(\cdot)$ starting at 
$z_0$ of System \eqref{nonlinear1} closed by the feedback control law 
\eqref{ST-feedback*} where the adaptive gains $L_1(t,z)$ and $L_2(t,z)$ are 
given in \eqref{eq:adaptgain33} and \eqref{eq:adaptgain4} respectively. 
Then,  $z(\cdot)$ is defined for all non-negative times, there exists a first time 
$\bar{t}$ for which $V(z(\cdot))\le{ \varepsilon/2}$. Then, for all $t\ge \bar{t}$, 
one has $V(z)<\varepsilon $. Moreover, for every trajectory $z(\cdot)$ of 
\eqref{eq:ST-feedback1}, there exist a time $t_0$ and two positive constants 
$Z_*<\varepsilon$ and $\xi_* >0$ depending on $z(\cdot)$, the bounds 
$\gamma_m$, $\psi_M$ and $\varepsilon$ such that, $V(z(t))\leq Z_*$ and 
$|\phi(t)|\leq \xi_*$ for $t\geq t_0$.
\end{montheo}
\begin{rem} As for Theorem~\ref{theorem1}, Remark~\ref{rem:unique} 
 holds true. Moreover, as it will be clear from the 
subsequent argument, the choice of the positive exponent of $L$ in 
$L_1(t,z)$ and $L_2(t,z)$ are not unique and can be adapted.
\end{rem}
 \subsection{Proof of Theorem~\ref{theorem1*}}Since the right-hand side of 
 \eqref{ST-feedback*} is continuous with respect to $(t,z,\phi)$, trajectories 
 starting from any $z_0,\phi_0$ at time $t=0$ exist on a non-trivial right 
 interval of zero.
 
 The first step of the argument consists in showing that $\bar{t}$ is well defined and finite for every trajectory of \eqref{eq:ST-feedback*}.
\begin{lemme}\label{le:existence2}
With the notations above, trajectories of \eqref{eq:ST-feedback*} are 
defined for all $t\leq \bar{t}$ and the latter is finite.
\end{lemme}
\begin{proof}
This result is a particular case off Lemma~\ref{le:existence1} in the case 
$\mu(\cdot)\equiv \varepsilon$ and  $\tilde{\phi}(t)=t^2+1$ and any choice of $l(\cdot)$ satisfying \eqref{eq:lt}.
 \end{proof}
 The next step consists in studying trajectories \eqref{eq:ST-feedback*} 
 inside ${\cal{V}}_\varepsilon$, while starting at $\bar{z}=z(\bar{t})$ 
 where $V(\bar{z})=\varepsilon/2$. Such a trajectory is defined on 
 a non-trivial interval $[\bar{t},T)$, with $T\leq \infty$.
 
 To proceed, we rely on a change of the state variable $z$.
 More precisely, we consider 
 \begin{equation}\label{eq:z-to-y}
 y:=\delta_{L^{1/2}(z)}(z),\quad 
 \end{equation}
 where we have dropped the index $\varepsilon$ in $L$.
 We next transform \eqref{eq:ST-feedback*} as follows.
\begin{lemme}\label{le:dynY}
Inside ${\cal{V}}_\varepsilon$, if $(z,\phi)$ is a trajectory of 
\eqref{eq:ST-feedback*}, then $(y,\phi)$ with $y=\delta_{L^{1/2}(z)}(z)$ is a trajectory of the following dynamical system
\beqnum\label{eq:ST-feedbackY}
\begin{array}{ccl}
L^{\frac\kappa2}\dot y(t)&=& Q+\frac1{2\varepsilon}\langle \nabla V(y),Q\rangle D_py(t),\\
L^{\frac\kappa2}\dot \phi(t)& =& -L^{\frac{p_{r+1}}2}\partial_r V(y(t))+
L^{\frac\kappa2}\dot\psi(t),\\
\end{array}
\eeqnum
where 
$$
L=L(z(t)),\quad Q=J_ry(t)+\gamma\ e_r\Big(\bar{c}L^{-\kappa/2}u_r(y(t))+L^{\frac{p_{r+1}}2} \phi(t)\Big).
$$
\end{lemme}
\begin{proof} By taking the time derivative on \eqref{eq:z-to-y}, one gets that
\begin{equation}\label{eq:dynY2}
\dot y=\delta_{L^{1/2}}(\dot z)+\frac{\langle \nabla L,\dot z\rangle}{2L}\ D_py.
\end{equation}
Using the various homogeneity properties given in Lemma~\ref{lem:basic1}, one gets
\begin{eqnarray}
\delta_{L^{1/2}}(\dot z)&=&\delta_{L^{1/2}}J_r(\delta_{L^{-1/2}}(y))+\gamma
\delta_{L^{1/2}}(e_r)\Big(\bar{c}L^{-\kappa/2}u_r(\delta_{L^{-1/2}}(y))+\phi\Big)\nonumber\\
&=&L^{-\kappa/2}J_ry+\gamma L^{p_r/2}e_r\Big(\bar{c}L^{-\kappa/2-p_{r+1}/2}u_r(y)+\phi\Big)\nonumber\\
&=&L^{-\kappa/2}\Big(J_ry+\gamma e_r\big(\bar{c}L^{-\kappa/2}u_r(y)+L^{\frac{p_{r+1}}2}\phi\big)\Big).\label{eq:dynY10}
\end{eqnarray}
On the other hand, it holds 
\begin{eqnarray}
\frac{\langle \nabla L,\dot z\rangle}{2L}&=&
\frac{\varepsilon}{2(\varepsilon-V(z))^2L}\langle \nabla V(z),\dot z\rangle
\nonumber\\
&=&\frac{L}{2\varepsilon}\langle \nabla V(\delta_{L^{-1/2}}(y)),\dot z\rangle
=\frac1{2\varepsilon}\langle \nabla V(y),\delta_{L^{1/2}}(\dot z)\rangle.
\label{eq:dynY20}
\end{eqnarray}
Putting together \eqref{eq:dynY10} and \eqref{eq:dynY20} in
\eqref{eq:dynY2}, one gets the first equation of \eqref{eq:ST-feedbackY}.

As for the dynamics of $\psi$, one has after using the fact that $\partial_r V$
is homogeneous of degree $2-p_r$ with respect to 
$(\delta_\varepsilon)_{\varepsilon}$
\begin{eqnarray*}
\dot \phi&=&-L\partial_rV(\delta_{L^{-1/2}}(y))+\dot\psi\\
&=&-L^{p_r/2}\partial_rV(y)+\dot\psi,
\end{eqnarray*}
which yields the second equation of \eqref{eq:ST-feedbackY}.
 \end{proof}
 For simplicity, we denote $V(y)$ by 
$Y$ and drop the (obvious) arguments in several quantities below. In order to determine the dynamics of $Y$, we need the following.
\begin{lemme}\label{lem:L-Y}
With the previous notations, one has that
\begin{equation}\label{eq:L-Y}
L=1+\tfrac{Y}\varepsilon.
\end{equation}
\end{lemme}
 \begin{proof}
 Recall that
 \begin{eqnarray*}
 L=\frac{\varepsilon}{\varepsilon-V(z)}=\frac{\varepsilon}{\varepsilon-V(\delta_{L^{-1/2}}(y))}=\frac{\varepsilon}{\varepsilon-L^{-1}V(y)},
 \end{eqnarray*}
 from which \eqref{eq:L-Y} follows.
  \end{proof}
  We can now compute the time derivative of $Y$ along trajectories of \eqref{eq:ST-feedbackY}.
  \begin{lemme}\label{lem:dotY}
With the previous notations, one has that
\begin{eqnarray}
L^{\frac\kappa2}\frac{d}{dt}\Big(\varepsilon\ln(\varepsilon+Y)\Big)&=&
-\rho Y^{1+\frac\kappa2}-\gamma\vert u_r(y)\partial_rV(y)\vert
(\bar{c}L^{-\kappa/2}-\tfrac1{\gamma})\nonumber\\
&+&\gamma L^{p_{r+1}}\partial_rV(y)\phi.\label{eq:dotY}
\end{eqnarray}
\end{lemme}
 \begin{proof} From the first equation of \eqref{eq:ST-feedbackY}, one gets
 $$
 L^{\frac\kappa2}\dot Y=\langle \nabla V(y),Q\rangle\Big(1+\frac{\langle \nabla V(y),D_py\rangle}{2\varepsilon}\Big)=\langle \nabla V(y),Q\rangle(1+\tfrac{Y}\varepsilon),
 $$
 where we have used \eqref{eq:grad}. One then concludes easily.
  \end{proof}
  Using the fact that $\gamma$ is constant, we prove that trajectories of
  \eqref{eq:ST-feedback*} inside ${\cal{V}}_\varepsilon$ are well defined for 
  all times $t\geq \bar{t}$.
\begin{lemme}\label{lem:L-Y1}
With the notations above, trajectories of \eqref{eq:ST-feedback*} are 
defined for all $t\geq \bar{t}$.
 \end{lemme}
 \begin{proof} If $z(\cdot)$ is a trajectory of \eqref{eq:ST-feedback*} inside 
 ${\cal{V}}_\varepsilon$, then, using \eqref{eq:dotY} and \eqref{eq:ST-feedbackY}, one gets as long as $z(\cdot)$ is defined, that
 \begin{eqnarray}\label{eq:lyap0}
 &L^{\frac\kappa2}&\frac{d}{dt}\Big(\varepsilon\ln(\varepsilon+Y)+\frac{\gamma_m}2\phi^2\Big)\leq -c_rY^{1+\frac\kappa2}\nonumber\\
 &-&\gamma_m\vert u_r(y)\partial_rV(y)\vert(\bar{c}L^{-\kappa/2}-
 \tfrac1{\gamma_m})+L^{\frac\kappa2}\gamma_m\psi_M\vert \phi\vert.\label{eq:lyap0}
 \end{eqnarray}
 Hence, if $L^{-\kappa/2}\geq \frac1{\gamma_m\bar{c}}$, the right-hand side of the above inequality is less than or equal to $L^{\frac\kappa2}\gamma_m\psi_M\vert \phi\vert$ and otherwise $Y$ is bounded by some $C_2(\gamma)$. In any case, after setting
 $$
 F(Y,\phi)=\varepsilon\ln(\varepsilon+Y)+\frac{\gamma_m}2\phi^2,
 \quad (Y,\phi)\in \R_+^*\
 $$
one deduces that
 \begin{eqnarray}
 \frac{d}{dt}F&\leq& C_2(\gamma,z(\cdot))+\gamma_m\psi_M\vert \phi\vert
 \nonumber\\
  &\leq&C_3(\gamma,z(\cdot))(1+F)^{1/2},\label{eq:defined}
 \end{eqnarray}
 as long as the trajectory $(y,\phi)$ of \eqref{eq:ST-feedbackY} is defined.
 Clearly \eqref{eq:defined} yields that 
 $
 F(Y(t),\phi(t))\leq C_4(\gamma,,z(\cdot))(1+t^2),
 $ 
 as long as that trajectory is defined for $t\geq \bar{t}$. It immediately follows that $(y,\phi)$ does not blow up in finite time $t$. This concludes the proof of the lemma.
 \end{proof}
 We next prove the last part of Theorem~\ref{theorem1*}, i.e. we get the following.
 \begin{lemme} Let $z(\cdot)$ be a trajectory of \eqref{eq:ST-feedback1}, 
 inside ${\cal{V}}_\varepsilon$. Then, there exist a time $t_0$ and two positive constants 
$Z_*<\varepsilon$ and $\xi_* >0$ depending on $z(\cdot)$, the bounds 
$\gamma_m$, $\psi_M$ and $\varepsilon$ such that, $V(z(t))\leq Z_*$ and 
$|\phi(t)|\leq \xi_*$ for $t\geq t_0$.
\end{lemme}
 \begin{proof} 
Set 
$
 G(Y,\phi)=\varepsilon\ln(\varepsilon+Y)+\frac{\gamma_m}2\phi^2-2\gamma_m\psi_Ms(\phi)z_r,
 $
 where $s(\cdot)$ is an odd $C^1$ function of saturation type, i.e. non-decreasing, $\xi s(\xi)>0$ for nonzero $\xi$, $s'(0)>0$, $s(\xi)=\xi$ for $0\leq \xi\leq 1/2$ and $s(\xi)=1$ for $\xi\geq 1$. Note that one can assume that $\vert s(\cdot)\vert$ and $s'(\cdot)$ are bounded by one and to respectively. We will also use below the fact that $\vert z_r\vert$ is bounded by $\varepsilon$.
 
 Using \eqref{eq:ST-feedback*}, \eqref{eq:L-Y} and  \eqref{eq:dotY}, it holds for 
 every time $t\geq \bar{t}$  that
\begin{eqnarray}
L^{\frac\kappa2}\frac{d}{dt}G&\leq& -c_rY^{1+\frac\kappa2}+2\gamma_m\psi_M\bar{c}
 (1+\varepsilon)\nonumber\\&-&\gamma_m\vert u_r(y)\partial_rV(y)\vert(\bar{c}L^{-\kappa/2}
 -\tfrac1{\gamma_m})
 \nonumber\\&-&L^{\frac\kappa2}\gamma_m\psi_M(2\phi s(\phi)-\vert \phi\vert)
 \nonumber\\&+&2\gamma_m\psi_M\vert z_r\vert s'(\phi)\big(L^{\frac{p_{r+1}}2}\vert \partial_r V(y(t))\vert+
L^{\frac\kappa2}\psi_M\big).
 \label{eq:lyap1}
 \end{eqnarray}
 
 By taking into account several bounds on $\vert s(\cdot)\vert$, $s'(\cdot)$
 and $\vert z_r\vert$, one deduces after dividing \eqref{eq:lyap1} by $L^{\frac\kappa2}$ that there exist positive constants $C_j:=C_j(\gamma,\phi,z(\cdot))$, $3\leq j\leq 4$, so that, for every time
 $t\geq \bar{t}$, it holds 
 \begin{equation}\label{eq:G}
 \frac{d}{dt}G\leq -\tfrac{\varepsilon c_r}2Y-\tfrac{\gamma_m\psi_M}2\vert\phi\vert+C_4.
  \end{equation}
  In particular, one gets from the above that there exists $C_5:=C_5(\gamma,
  \phi)>0$ so that $G$ is decreasing along trajectories of \eqref{eq:ST-feedbackY}
 as soon as $G>C_5$. Since $G$ is bounded below by 
 $\varepsilon(\ln(\varepsilon)+\gamma_m\psi_M)$, one deduces that $G$ is bounded along any trajectory of \eqref{eq:ST-feedbackY}, concluding the proof of the lemma.
\end{proof}
  
\subsection{An example}

In order to point out the fact that the adaptive HOST algorithm \eqref{ST-feedback*} overcomes the problem of unbounded gain function $L(t,z)$, let us consider the same third order disturbed system with the same  function $\phi(t)$, however $\gamma$ is chosen as a constant $\gamma=2$ instead of a positive function. Here, the adaptive HOST algorithm \eqref{ST-feedback*} is employed where $L_1(t,z)$ and $L_2(t,z)$ are chosen to satisfy \eqref{eq:adaptgain33}-\eqref{eq:adaptgain4}, the AHFP is taken as in the previous algorithm, and $\varepsilon=0.1$.

\begin{figure}[h!]
\includegraphics[trim= 0.5cm 2.3cm 0.1cm 3.0cm, clip, width=9cm]{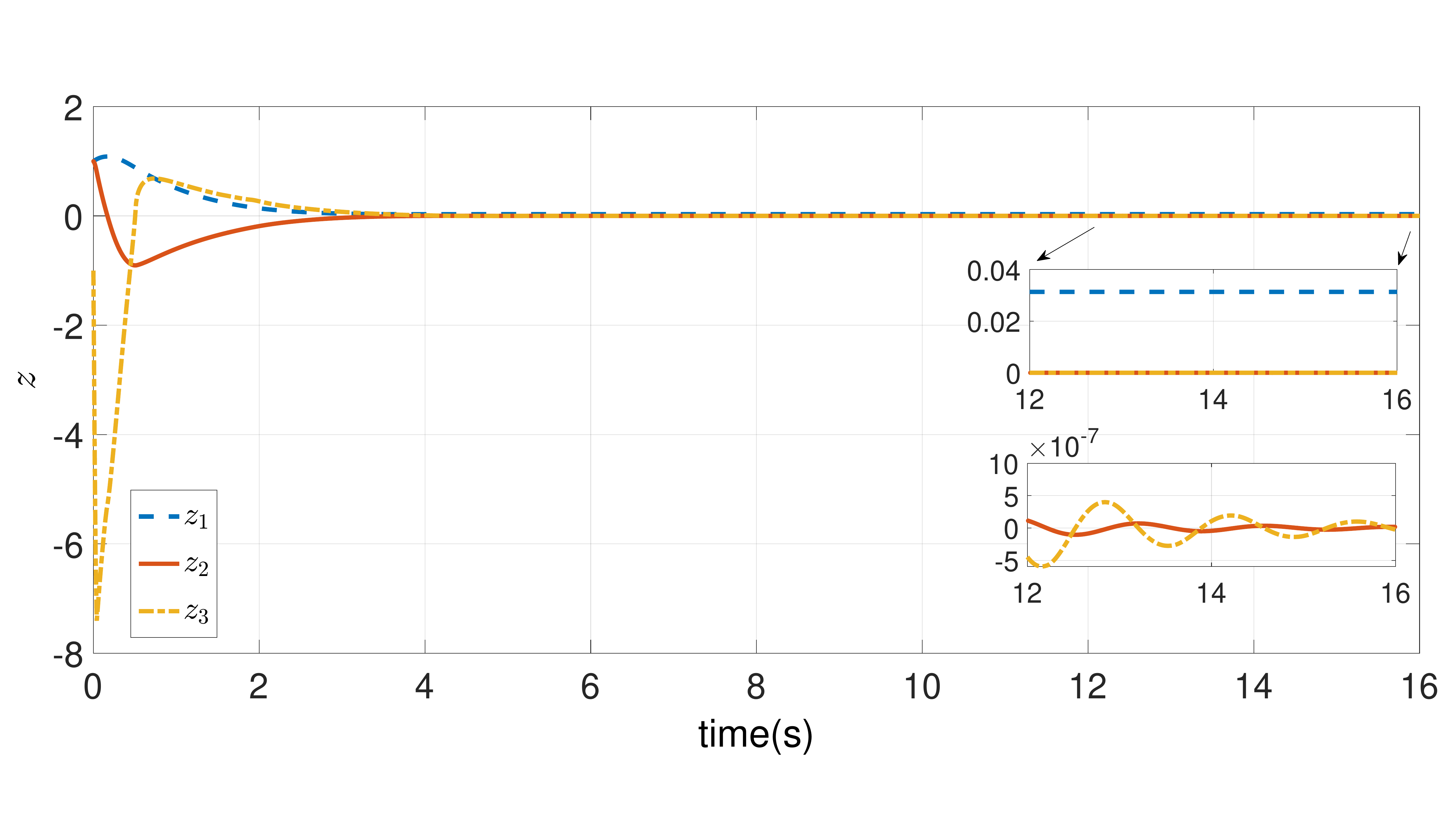}%
\caption{ The states $z_1$, $z_2$ and $z_3$ with the adaptive HOST algorithm \eqref{ST-feedback*}.}
\label{fig:statescase21}
\end{figure}

\begin{figure}[h!]
\includegraphics[trim= 0.5cm 2.3cm 0.1cm 3.0cm, clip, width=9cm]{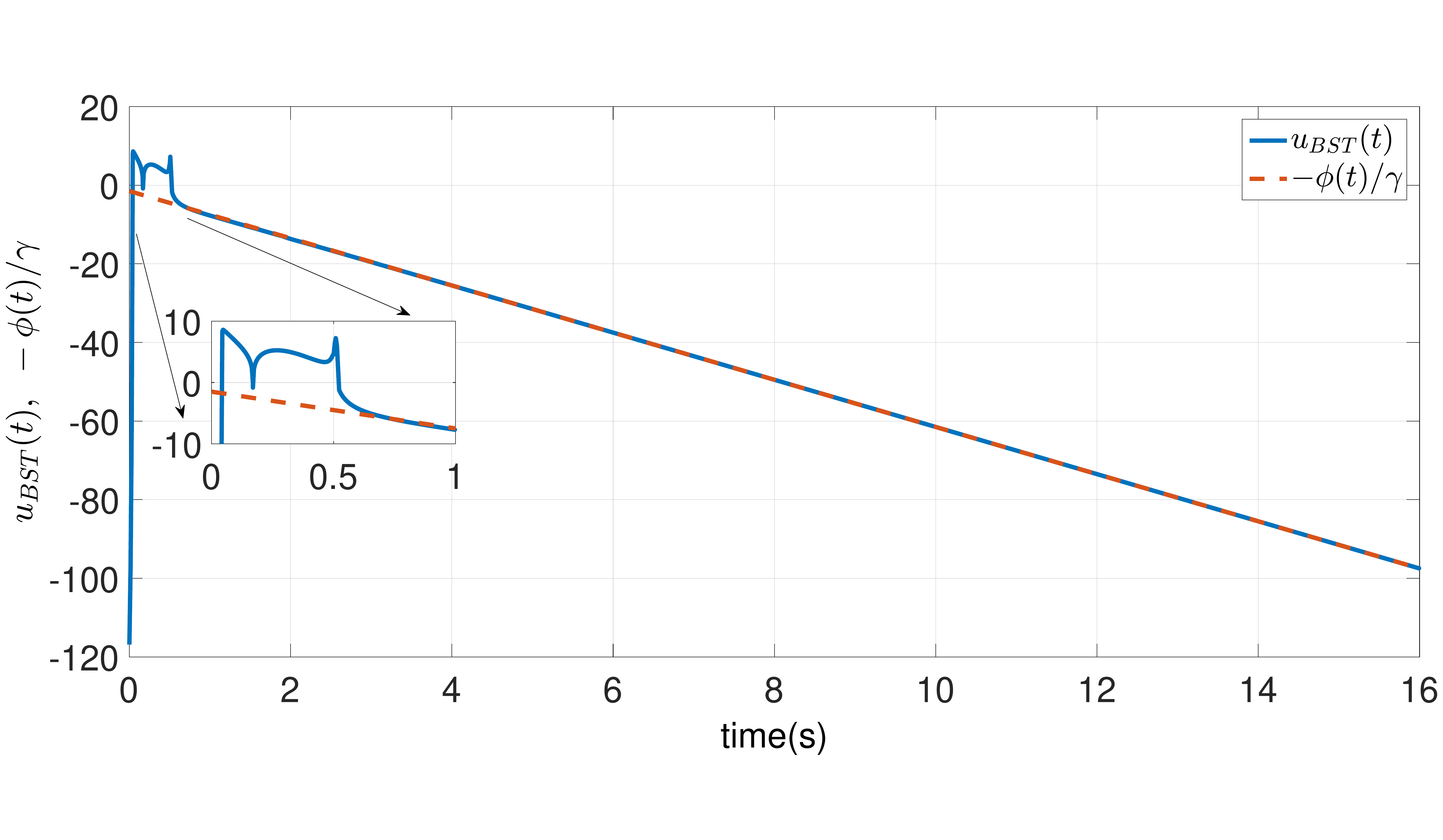}%
\caption{Evolution of the control signal $u_{BST}(t)$ and the uncertainties $-\phi(t)/\gamma$.}
\label{fig:estimcase21}
\end{figure}

\begin{figure}[h!]
\includegraphics[trim= 2.5cm 2.3cm 4.1cm 3.0cm, clip, width=9cm]{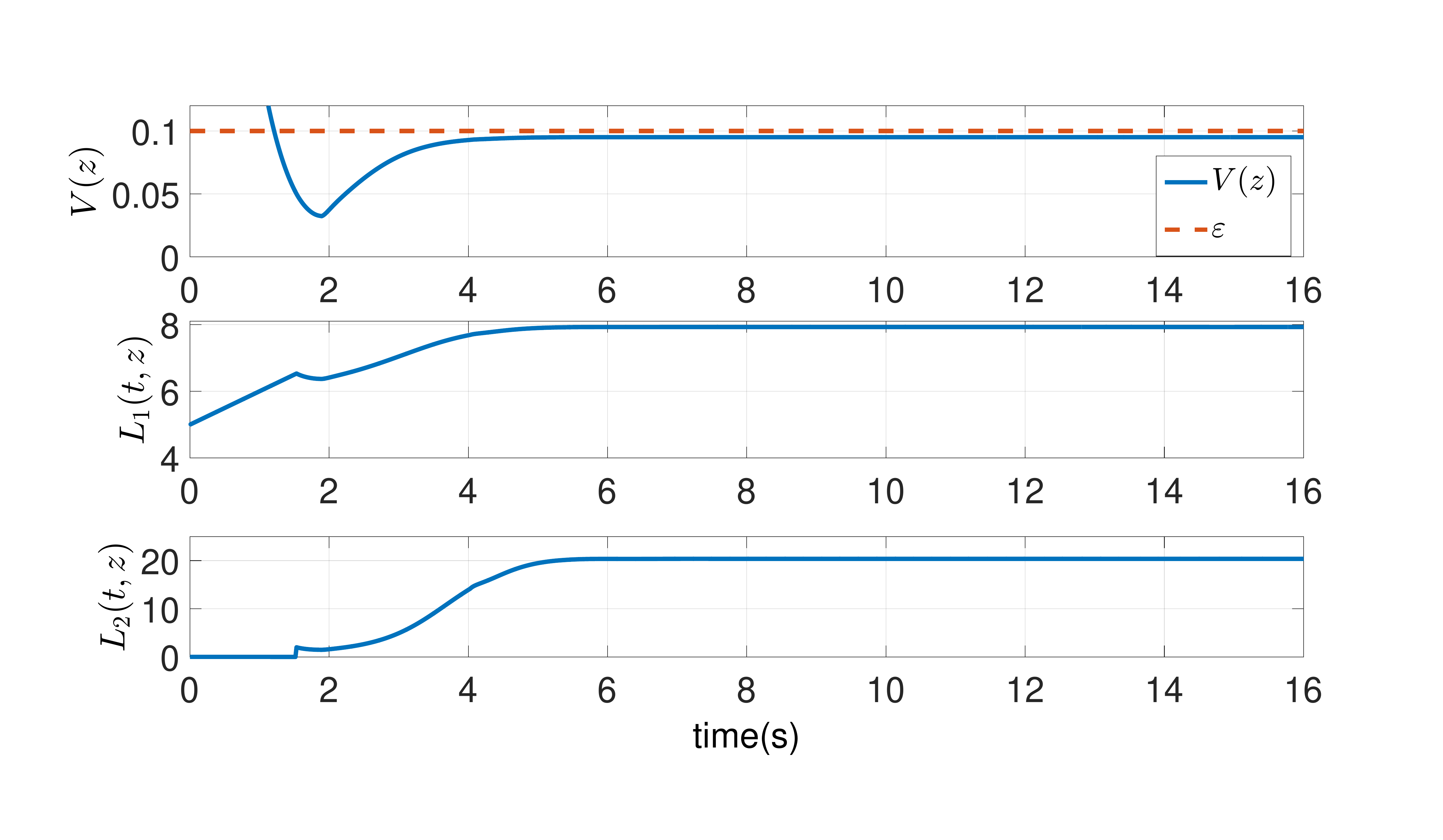}%
\caption{Evolution of the Lyapunov function $V(z)$ and the adaptive gains $L_1(t,z)$ and $L_2(t,z)$ using the adaptive HOST algorithm \eqref{ST-feedback*}.}
\label{fig:lyapunovcase21}
\end{figure}

\begin{figure}[h!]
\includegraphics[trim= 2.5cm 2.3cm 4.1cm 3.0cm, clip, width=9cm]{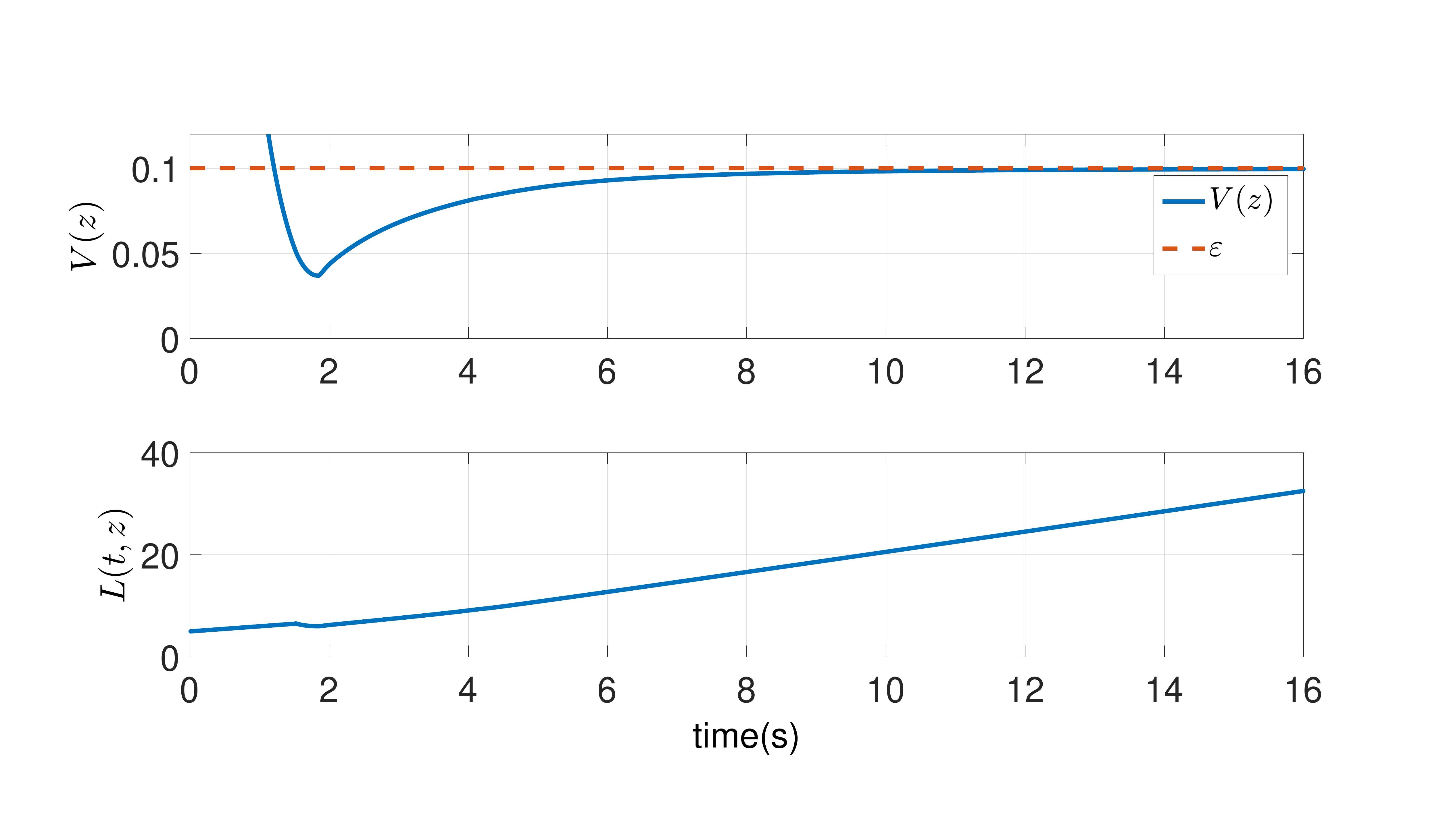}%
\caption{Evolution of the Lyapunov function $V(z)$ and the adaptive gains $L(t,z)$ using the adaptive controller \eqref{ST-feedback}.}
\label{fig:lyapunovcase11}
\end{figure}

The results are shown in Figs.~\ref{fig:statescase21}-\ref{fig:estimcase21}-\ref{fig:lyapunovcase21}. In Fig.~\ref{fig:statescase21}, it can be seen that the state variables $z_1$, $z_2$ and $z_3$ converge to some neighborhoods of zero. Fig.~\ref{fig:estimcase21} shows that the adaptive controller $u_{BST}(t)$ is continuous and it follows the uncertainties $-\phi(t)/\gamma$. Figs.~\ref{fig:lyapunovcase21}-\ref{fig:lyapunovcase11} illustrate the evolution of the Lyapunov functions and the adaptive gains using the adaptive HOST controller \eqref{ST-feedback*} and the adaptive controller \eqref{ST-feedback}. It can be noticed that for both algorithms, the Lyapunov function $V(z)$ stays in the predefined vicinity $V(z)<\varepsilon$ starting from the time instant $\bar{t}$. However, Fig.~\ref{fig:lyapunovcase21} shows that the adaptive gains $L_1(t,z)$ and $L_2(t,z)$ are bounded, which is not the case using the adaptive controller \eqref{ST-feedback} (see Fig.~\ref{fig:lyapunovcase11}). Hence, the specific feature of the adaptive HOST controller \eqref{ST-feedback*} is that it does not allow the adaptive gains to tend to infinity even if the uncertainties are unbounded.

\section{Conclusion}\label{sec:conclusion}

In this paper, two new barrier function-based adaptive CHOSMCs are developed for perturbed chains of integrators with unbounded perturbation. The first class can be viewed as a continuous generalization of the adaptive DHOSMC in the case of unbounded perturbation. In the general case of unbounded perturbation, the gains can grow to infinity. 
A second class of adaptive controllers refer to as adaptive HOST algorithm is developed for the case of Lipschitz unbounded perturbations. This second class ensures the boundedness of 
adaptive controllers’ gains.

\bibliography{BFBA-HOST_210622_bib}
\end{document}